\documentclass[a4j,11pt,dvipdfmx]{article}
\usepackage{amscd}
\usepackage{amsmath}
\usepackage{amssymb}
\usepackage{amsthm}
\usepackage{ascmac}
\usepackage{graphicx}
\usepackage{multicol}
\theoremstyle{definition}
\newtheorem{df}{Definition}
\newtheorem{prop}{Proposition}
\newtheorem{thm}{Theorem}
\title{Normal-sized hypercuboids in a given hypercube}
\author{Takashi HIROTSU}
\date{\today}
\begin{document}
\maketitle
\begin{abstract}
In a given hypercube, draw grid lines parallel to the edges, and consider all hypercuboids (or hypercubes) whose edges are lying on the grid lines or the boundary. 
We find the limit of the value of the ratio of the arithmetic mean of the volumes of those hypercuboids (or hypercubes) to the entire volume as the grid spacing becomes smaller. 
\end{abstract}
The author posted the following problem on his own website \cite{3} for high school students: ``On a square, draw grid lines to divide each edge into $m$ segments, and consider all squares surrounded by grid lines or the boundary. 
Find the limit of the value of the ratio of the arithmetic mean of the areas of those squares to the entire area as $m \to \infty.$'' 
The answer is $1/10.$ 
To clarify why the number $10$ appears, we generalize this problem to higher dimensions.
\begin{df}\label{Df1}
Let $n$ and $m$ be positive integers. 
In an $n$-dimensional hypercube $H \subset \mathbb R^n$ of edge length $a,$ draw grid lines to divide each edge into $m$ segments, and consider all hypercuboids whose edges are lying on the grid lines or the boundary. 
Denote by $q_n(m)$ the value of the ratio of the arithmetic mean of the volumes of those hypercuboids to the entire volume, and put $q_n = \lim_{m \to \infty}q_n(m).$ 
We say that a hypercuboid is {\itshape normal-sized} if its volume is $q_na^n.$\par
Similarly, denote by $r_n(m)$ the value of the ratio of the arithmetic mean of the volumes of such hypercubes to the entire volume, and put $r_n = \lim_{m \to \infty}r_n(m).$ 
We say that a hypercube in $H$ is {\itshape normal-sized} if its volume is $r_na^n.$
\end{df}
\begin{thm}\label{Thm1}
In an $n$-dimensional hypercube $H \subset \mathbb R^n,$ the value $q_n$ of the ratio of the volume of a normal-sized hypercuboid to the entire volume is 
\[ q_n = \frac{1}{3^n}.\]
\end{thm}
\begin{proof}
It suffices to show for the case $H = [0,1]^n.$ 
Since the number of small hypercuboids of edge lengths $j_1/m,$ $\cdots,$ $j_n/m$ in $[0,1]^n$ is $(m+1-j_1)\cdots (m+1-j_n)$ for $1 \leq j_1 \leq m,$ $\cdots,$ $1 \leq j_n \leq m,$ the value of the ratio of the arithmetic mean of the volumes of all small hypercuboids to the entire volume is 
\begin{align*} 
q_n(m) &= \frac{\sum_{j_1 = 1}^m\cdots\sum_{j_n = 1}^m(m+1-j_1)\cdots (m+1-j_n)(j_1/m)\cdots (j_n/m)}{\sum_{j_1 = 1}^m\cdots\sum_{j_n = 1}^m(m+1-j_1)\cdots (m+1-j_n)} \\ 
&= \frac{\left(\sum_{i = 1}^m(m+1-i)i\right) ^n}{m^n\left(\sum_{i = 1}^mi\right) ^n} = \frac{(m(m+1)(m+2)/6)^n}{m^n(m(m+1)/2)^n} = \frac{(m+2)^n}{3^nm^n}. 
\end{align*} 
By taking the limit, we obtain 
\[ q_n = \lim\limits_{m \to \infty}q_n(m) = \lim\limits_{m \to \infty}\frac{1}{3^n}\left( 1+\dfrac{2}{m}\right) ^n = \frac{1}{3^n}.\]
\end{proof}
\begin{thm}\label{Thm2}
In an $n$-dimensional hypercube $H \subset \mathbb R^n,$ the value $r_n$ of the ratio of the volume of a normal-sized hypercube to the entire volume is 
\[ r_n = \frac{1}{\binom{2n+1}{n}}.\]
\end{thm}
\begin{proof}
It suffices to show for the case $H = [0,1]^n.$ 
Since the number of small hypercubes of edge length $j/m$ in $[0,1]^n$ is $(m+1-j)^n$ for $1 \leq j \leq m,$ the value of the ratio of the arithmetic mean of the volumes of all small hypercubes to the entire volume is 
\[ r_n(m) = \frac{\sum_{j = 1}^m(m+1-j)^n(j/m)^n}{\sum_{j = 1}^m(m+1-j)^n} = \frac{\sum_{i = 0}^n(-1)^i\binom{n}{i}(m+1)^{n-i}\sum_{j = 1}^mj^{n+i}}{m^n\sum_{i = 1}^mi^n}\] 
by the binomial theorem. 
Since the leading term of $(m+1)^{n-i}\sum_{j = 1}^mj^{n+i},$ $m^n\sum_{i = 1}^mi^n$ in $m$ are $m^{2n+1}/(n+1+i),$ $m^{2n+1}/(n+1)$ respectively, we have 
\[ r_n = \lim\limits_{m \to \infty}r_n(m) = (n+1)\sum_{i = 0}^n\frac{(-1)^i}{n+1+i}\binom{n}{i} = \frac{n+1}{(n+1)\binom{2n+1}{n}} = \frac{1}{\binom{2n+1}{n}}\]
by the following proposition.
\end{proof}
\begin{prop}\label{Prop1}
We have 
\[\sum_{i = 0}^n\frac{(-1)^i}{n+1+i}\binom{n}{i} = \frac{1}{(n+1)\binom{2n+1}{n}}.\]
\end{prop}
\begin{proof}
By the binomial theorem, we have 
\[\sum_{i = 0}^n\binom{n}{i}x^{n+i} = x^n(1+x)^n.\] 
By integrating the both sides from $-1$ to $0,$ we obtain 
\[\sum_{i = 0}^n\binom{n}{i}\int_{-1}^0x^{n+i}dx = \int_{-1}^0x^n(1+x)^ndx.\] 
On the left hand side, we have 
\[\sum_{i = 0}^n\binom{n}{i}\int_{-1}^0x^{n+i}dx = \sum_{i = 0}^n\binom{n}{i}\left[\frac{x^{n+1+i}}{n+1+i}\right]_{-1}^0 = (-1)^n\sum_{i = 0}^n\frac{(-1)^i}{n+1+i}\dbinom{n}{i}.\] 
On the right hand side, by putting $t = -x$ and applying the relation between the Beta and Gamma functions $\mathrm B(x,y) = \Gamma (x)\Gamma (y)/\Gamma (x+y)$ to $x = y = n+1,$ we obtain 
\begin{align*} 
\int_{-1}^0x^n(1+x)^ndx &= \int_1^0(-t)^n(1-t)^n(-1)dt = (-1)^n\int_0^1t^n(1-t)^ndt \\ 
&= (-1)^n\frac{n!n!}{(2n+1)!} = (-1)^n\frac{n!(n+1)!}{(n+1)\cdot (2n+1)!} \\ 
&= (-1)^n\frac{1}{(n+1)\binom{2n+1}{n}}. 
\end{align*} 
These imply the desired identity.
\end{proof}
\begin{thm}\label{Thm3}
In an $n$-dimensional hypercube $H \subset \mathbb R^n$, the value $\sqrt[n]{r_n}$ of the ratio of the edge length of a normal-sized hypercube to that of $H$ converges to $1/4$ as $n \to \infty.$
\end{thm}
\begin{proof}
By Theorem \ref{Thm2}, we have 
\[\frac{1}{\sqrt[n]{r_n}} = \sqrt[n]{\binom{2n+1}{n}}.\] 
By the following proposition, we obtain the desired result.
\end{proof}
\begin{prop}\label{Prop2}
We have 
\[\lim\limits_{n \to \infty}\sqrt[n]{\binom{2n}{n}} = \lim\limits_{n \to \infty}\sqrt[n]{\binom{2n+1}{n}} = 4.\] 
\end{prop}
\begin{proof}
By the binomial theorem, we have $(1+1)^{2n} = \sum_{i = 0}^{2n}\binom{2n}{i},$ and therefore 
\[\frac{4^n}{2n+1} \leq \binom{2n}{n} \leq 4^n.\] 
We also have 
\[\binom{2n+1}{n} = \frac{2n+1}{n+1}\binom{2n}{n}.\] 
By taking the $n$-th roots and the limits, we obtain the desired formulae, since 
\[\lim\limits_{x \to \infty}(ax+b)^{\frac{1}{x}} = 1 \quad (a,\ b > 0),\] 
which follows from 
\[\lim\limits_{x \to \infty}\log(ax+b)^{\frac{1}{x}} = \lim\limits_{x \to \infty}\frac{\log (ax+b)}{x} = \lim\limits_{x \to \infty}\frac{a/(ax+b)}{1} = 0\] 
by l'H\^{o}pital's rule.
\end{proof}

\end{document}